\documentclass[a4paper,reqno, 11pt]{amsart}
\usepackage{mathrsfs}
\usepackage{amssymb}
\usepackage[usenames,dvipsnames,svgnames,table]{xcolor}
\usepackage{url}

\usepackage[T1]{fontenc}

\makeatletter
\def\thm@space@setup{
\thm@preskip=2mm
\thm@postskip=0mm
}
\makeatother

\usepackage{boldline}

\let\leq\leqslant

\let\subset\subseteq

\let\epsilon\varepsilon

\usepackage[centertags]{amsmath}
\usepackage{amsfonts}
\usepackage{amsthm}
\usepackage{times}

\usepackage{hyperref}
\usepackage[capitalise, noabbrev]{cleveref}

\hypersetup{
    colorlinks,
    linkcolor={RoyalBlue},
    citecolor={OliveGreen},
    urlcolor={blue!80!black}
}

\usepackage{enumerate} %
\usepackage{xcolor}
\usepackage[margin=1.25in]{geometry}
\usepackage{centernot}
\usepackage{csquotes}
\usepackage[symbol]{footmisc}
\usepackage{todonotes}

\usepackage{tabularx}

\definecolor{myblue1}{RGB}{30, 144, 255}
\definecolor{myblue2}{RGB}{0, 135, 245}
\definecolor{myblue3}{RGB}{0, 126, 235}
\definecolor{myblue4}{RGB}{36, 70, 119}

\newtheorem{theorem}{Theorem}
\newtheorem*{theorem*}{Theorem}

\newtheorem{corollary}[theorem]{Corollary}
\newtheorem{lemma}[theorem]{Lemma}

\newtheorem{problem}[theorem]{Problem}

\theoremstyle{definition}

\newcommand{\calS}{\mathcal{S}}

\newcommand{\calT}{\mathcal{T}} 
\newcommand{\calR}{\mathcal{R}}

\newcommand{\N}{\mathbb{N}}

\newcommand{\Z}{\mathbb{Z}}

\usepackage{mathtools}
\DeclarePairedDelimiter\set{\{}{\}}

\begin{document}

	\title{Counting Unions of Schreier Sets}

	\author{Kevin Beanland}
	\author{Dmitriy Gorovoy}
	\author{J\c{e}drzej Hodor}
	\author{Daniil Homza}

	\address{(K. Beanland) Department of Mathematics, Washington and Lee University, Lexington, VA 24450.}
	\email{beanlandk@wlu.edu}
	\address{(D. Gorovoy, J. Hodor, D. Homza)  Mathematics Department,
  Faculty of Mathematics and Computer Science, Jagiellonian University, Krak\'ow, Poland.}
	\email{dimgor2003@gmail.com,jedrzej.hodor@gmail.com,daniil.homza.work@gmail.com}
	
	\thanks{2010 \textit{Mathematics Subject Classification}. Primary: }
	\thanks{\textit{Key words}: }
	
	\thanks{J.\ Hodor is partially supported by a Polish National Science Center grant (BEETHOVEN; UMO-2018/31/G/ST1/03718).}
	
	
	\begin{abstract}
    A subset of positive integers $F$ is a Schreier set if it is non-empty and $|F|\leqslant \min F$ (here $|F|$ is the cardinality of $F$). 
    For each positive integer $k$, we define $k\calS$ as the collection of all the unions of at most $k$ Schreier sets.
    Also, for each positive integer $n$, let $(k\calS)^n$ be the collection of all sets in $k\calS$ with the maximum element equal to $n$.
    It is well-known that the sequence $(|(1\calS)^n|)_{n=1}^\infty$ is the Fibbonacci sequence.
    In particular, the sequence satisfies a linear recurrence.
    We generalize this statement, namely, we show that the sequence $(|(k\calS)^n|)_{n=1}^\infty$ satisfies a linear recurrence for every positive $k$.
  \end{abstract}
	
	\maketitle
	
	
\section{Introduction}

    A subset $F$ of natural numbers\footnote{We write $\mathbb{N}$ for the set of all positive integers, and we refer to members of $\mathbb{N}$ as natural numbers.} is called a \emph{Schreier set} if it is non-empty and $|F| \leq \min F$ (here $|F|$ is the cardinality of the set $F$). Let $\calS$ denote the family of all Schreier sets and $\mathcal{S}^n$ be the set of Schreier sets $F$ with $\max F=n$. The family $\calS$, which we call the \emph{Schreier family}, was introduced in the 1930s in a seminal paper of J\'ozef Schreier \cite{Schreier1930}. In this paper, Schreier solved a problem in Banach space theory posed to Banach and Saks. It is now well-established that Schreier sets and their generalizations have deep connections to the norm convergence of convex combinations of weakly null sequences  \cite{Alspach-Argyros, AT-Book, Odell-Ordinal} and are thus of fundamental importance in Banach space theory.
    
    In this paper, we focus on the combinatorial properties of the Schreier sets. Recently, counting various generalizations of Schreier sets has received a lot of attention and can be seen by H.V. Chu and coauthors' work \cite{ Hung3, Hung4,  Hung5, Hung1}. Research in this direction began with the simple observation by an anonymous\footnote{The blogger was later to be revealed to be Alistair Bird who discovered this connection while writing his Ph.D. thesis in Banach space theory.} blogger that  $(|\mathcal{S}^n|)_{n=1}^\infty$ is the Fibonacci sequence.
    This fact follows from the observation that the set elements of $\mathcal{S}^n$ which contain $n-1$ can be put in bijection with $\mathcal{S}^{n-1}$, and the set of elements of $\mathcal{S}^n$ not containing $n-1$ can be put in bijection with  $\mathcal{S}^{n-2}$. 
    
    
    It is natural to ask if, for various modified Schreier families, the analogous sequence (counting members with a fixed maximal element) is a linear recurrence sequence. A natural and nontrivial modification of $\calS$ is to fix natural numbers $p$ and $q$ and consider the family 
    $$\calS_{p, q}=\{F \subset \mathbb{N}: q  \min F \geqslant p  |F|\}$$ 
    In \cite{Hung3}, Beanland, Chu, and Finch-Smith proved that the sequence $(|\calS^n_{p,q}|)_{n=1}^\infty$ is a linear recurrence sequence and gave a compact recursive formula (this builds on the earlier work \cite{Hung6}).
    
   In the current paper, we consider, for each natural number $k$, the set $k\calS$ of all subsets of $\mathbb{N}$ that can be written as a union of at most $k$ many Schreier sets. These sets were recently studied in a Banach space theory context in \cite{Beanland-Hodor} by the first and third authors of the present paper. 
   
   In the present work, we prove that for each $k$ the sequence $\{|(k\calS)^n|\}_{n=1}^\infty$ satisfies a specific recurrence relation that is generated by the Fibonacci recurrence in a natural way. 
In order to describe these recurrences we recall a few basic notions. A sequence $(a_n)_{n=1}^\infty$ satisfies the linear recurrence with coefficients $c_k,c_{k-1},\cdots,c_0$ (we assume that $c_k\not=0$) if for each $n \in \mathbb{N}$, we have
\[ c_k a_{n+k}+c_{k-1}a_{n+k-1} + \cdots + c_1 a_{n+1}+ c_0 a_n = 0\]
The {\em characteristic polynomial} of the linear recurrence is given by
\[ p(x)=c_kx^k +c_{k-1}x^{k-1} + \cdots +c_1 x +c_0.\]

We now introduce a sequence of polynomials $(p_k(x))_{k=1}^\infty$ that will enable us to state our main result. Let $p_0(x)=x-1$. This is the characteristic polynomial of the constant sequence. Suppose that $p_k(x)$ has been defined for some non-negative integer $k$, and let 
\begin{equation}
p_{k+1}(x):= p_k(x(x-1)). \label{defn:p_k}
\end{equation}
Below we have computed $p_k(x)$ for each $k \in \{1,2,3\}$.
\begin{enumerate}
    \item $p_1(x) = x(x-1)-1= x^2-x-1$, the characteristic polynomial of the Fibonacci recursion.
    \item $p_2(x) = x^2(x-1)^2-x(x-1) - 1= x^4-2x^3+x -1$.
    \item $p_3(x) = x^8 - 4 x^7 + 4 x^6 + 2 x^5 - 5 x^4 + 2 x^3 + x^2 - x - 1.$
\end{enumerate}
Notice that $p_k(x)$ has degree $2^k$. For all natural numbers $k$ and $n$, set $s_{k,n}:= |(k\calS)^n| $. We can now state our main theorem. 

    \begin{theorem} \label{thm:lin-rec}
         For each $k\in \mathbb{N}$, the sequence $(s_{k,n})_{n=1}^\infty$ satisfies the linear recurrence with characteristic polynomial $p_k(x)$.
    \end{theorem}
    \vspace{.1in}

 Unlike the proof in the $k=1$ case, our proof for general $k$ is quite involved and requires connecting the sequences $(s_{k-1,n})_{n=1}^\infty$ with $(s_{k,n})_{n=1}^\infty$ is a non-obvious but elegant way. In addition, we obtain an interesting corollary to the proof Theorem \ref{thm:lin-rec}. 
 In order to state it, we need some more notation.
 For two sets $E,F \subset \mathbb{N}$ we write $E<F$ if $\max E < \min F$.
 We say that sets $E_1,\dots,E_m \subset \mathbb{N}$ are \emph{successive} if $E_1 < E_2 < \dots < E_m$.
 A set $F \in \calS$ is called a \emph{maximal Schreier set} if $|F|=\min F$. 
 Furthermore, $F \in k\calS$ is a \emph{maximal $k$-Schreier set}, if it can be decomposed as the disjoint union of $k$-many successive maximal Schreier sets.
 Let $\mathcal{R}_{k,n}^0\subset (k\mathcal{S})^n$ be the collection of all maximal $k$-Schreier sets in $(k\mathcal{S})^n$ and $r^0_{k,n}:=|\mathcal{R}^0_{k,n}|$.

 \begin{corollary}
For each $k\in \mathbb{N}$, the sequence $(r^0_{k,n})_{n=1}^\infty$ satisfies a linear recurrence with characteristic polynomial $p_k(x)$. \label{only cor}
 \end{corollary}
    
\section{Notation and overview of the proof of Theorem \ref{thm:lin-rec}} \label{sec:int}    


Let $F\subset \mathbb{N}$. We will inductively define a subset $E_k(F)$ of $F$ for each $k \in \mathbb{N}$. Let $E_1(F)$ be the initial segment of $F$ that is a maximal Schreier set, and if no such set exists, let  $E_1(F)=F$. Assuming $E_k(F)$ has been defined for some positive integer $k$, let 
$$E_{k+1}(F) = E_1(F \setminus \bigcup_{i=1}^k E_i(F)).$$

\noindent Observe that $E_1(F), E_2(F),\ldots$ is a partition of $F$ into consecutive Schreier sets, where all non-empty sets are maximal except possibly for the last non-empty one. 

Let $k\in \mathbb{N}$ and $F \in k\calS$. Let $d_k(F)$ be the greatest non-negative integer $d$ so that there is a set $G\subset\mathbb{N}$ with $F<G$, $|G|=d$, and $F\cup G\in k\calS$. If no such $d$ exists, we let $d_k(F)=\infty$.
Note that 
\begin{enumerate}
    \item $d_k(F)<\infty$ if and only if $E_k(F)\not=\emptyset$.
    \item $F$ is a maximal $k$-Schreier set if and only if $d_k(F)=0$.
\end{enumerate} 
For each $n,k\in \mathbb{N}$ and $d$, we define
       \begin{equation}
         \calR_{k,n}^d := \{ F \in (kS)^n \ : \ d_k(F) = d \} \ \ \ \mbox{ and } \ \ \ r_{k,n}^d := |\calR_{k,n}^d|.
     \end{equation}
    Also, let $r_{k,n}^d := 0$ for all $k,n \in \Z, d \in \Z \cup\{\infty\}$ such that the value was not defined above. 
    As a visual aid see the initial values of $r_{k,n}^d$ for $k \in \{1,2\}$ in \cref{tab:k=1-2}.
   
       \newcolumntype{M}[1]{>{\centering\arraybackslash}m{#1}}
\newcolumntype{N}{@{}m{0pt}@{}}

\begin{table}[h!]
    \centering
    \hspace{2.5cm}
    \begin{tabular}{ cc }
        $k=1$ & $k=2$ \\
        \begin{tabular}{  M{0.5cm} V{2.7} M{0.3cm}| M{0.3cm} | M{0.3cm} | M{0.3cm} | M{0.3cm} | M{0.3cm} | M{0.3cm} | M{0.3cm} | M{0.3cm} } 
             $n\backslash d$ & $0$ & $1$ & $2$ & $3$ & $4$ & $5$ & $6$ & 7 & $\cdots$ \\
             \hlineB{2.7}
             1  & 1  &   &  &  & &  &  & &  \\
             \hline
              2  & 0  & 1 & & &  & &  &  & \\
             \hline
              3  & 1  & 0 & 1 &  &  &  &  & &  \\
             \hline
              4  & 1  & 1 & 0& 1 &  & & & & \\
             \hline
              5 & 2  & 1 & 1 & 0 & 1 & & & &  \\
             \hline
              6  & 3  & 2 & 1 & 1 & 0 & 1 &  & &  \\
             \hline
              7  & 5  & 3 & 2 & 1 & 1 & 0 & 1 & &  \\
             \hline
              8  & 8 & 5 & 3 & 2 & 1 & 1 & 0 & 1&   \\
             \hline
             $\vdots$  & $\vdots$  &  &  &  & &  &&  & $\ddots$ 
            \end{tabular} &  
        \begin{tabular}{  M{0.5cm} V{2.7} M{0.3cm}| M{0.3cm} | M{0.3cm} | M{0.3cm} | M{0.3cm} | M{0.3cm} | M{0.3cm} | M{0.3cm} | M{0.3cm} } 
             $n\backslash d$ & $0$ & $1$ & $2$ & $3$ & $4$ & $5$ & $6$ & 7 & $\cdots$ \\
             \hlineB{2.7}
             1 & 0  & &  &  & &  &  &  & \\
             \hline
              2  & 0  & 1 &  & &  & &  &   &\\
             \hline
              3  & 1  & 0 & 1 & &  &  &  &  &\\
             \hline
              4  & 1  & 1 & 0 & 2 & & & &  & \\
             \hline
              5  & 2  & 1 & 2 & 0 & 3 & & &   & \\
             \hline
              6  & 3  & 3 & 2 & 3 & 0 & 5 &  &   & \\
             \hline
              7  & 6  & 5 & 5 & 3 & 5 & 0 & 8 &   & \\
             \hline
              8  & 11 & 10 & 8 & 8 & 5 & 8 & 0 & 13  &  \\
             \hline
             $\vdots$  & $\vdots$  &  &  &  & &  &&  & $\ddots$
            \end{tabular} \\
        \end{tabular}
    \newline
    \caption{The left are the initial values of $(r_{1,n}^d)$, and the right are the initial values of $(r_{2,n}^d)$.}
    \label{tab:k=1-2}
\end{table}

We now record a few observations regarding the above tables. 
\begin{itemize}
    \item[(O1)] Each entry below the first two main diagonals is computed by adding the entry directly north to the northeast entry. We call this the Pascal-like property. Consequently, the two main diagonals determine the rest of the table. Note also that the second main diagonal is constantly zero; we will show this holds in general (Lemma \ref{lem:pascal-rec}).
    \item[(O2)] Due to the Pascal-like property of the tables, if the first two main diagonals satisfy the same recursion, each diagonal satisfies that recursion. Moreover, we can express each term of the diagonal in terms of the first column and use this to compute the recursion for the first column (Lemma \ref{lem:pascal-fm} and \cref{fig:weighted}).  
    \item[(O3)] In the $k=2$ table, the sequence starting at the second entry of the main diagonal is equal to the partial sums of the first column of $k=1$ (Lemma \ref{move tables}). Consequently, these sequences satisfy the exact same recursions. In this case, both sequences satisfy the Fibonacci linear recurrence.
\end{itemize}
   
Taken together, these observations outline an inductive process to compute all of the entries in the tables up to any given $k$. As we can see from the table $k=1$, the sequence $(r_{1,n}^0)_{n=1}^{\infty}$ satisfies the Fibonacci recursion. As we mentioned in the introduction, the sequence $(s_{1,n})_{n=1}^{\infty}$ also satisfies the Fibonacci recursion. This led to the conjecture that this holds for each $k$, namely, that $(r_{k,n}^0)_{n=1}^{\infty}$ and $(s_{k,n})_{n=1}^{\infty}$ satisfy the same recursion relation. We verify this in the affirmative. The main step is to note and prove that for all natural numbers $k,n$, we have
        \begin{equation}
            s_{k,n} = 2s_{k,n-1} - r_{k,n-1}^0. \label{bigdeal}
        \end{equation} 

\section{Initial values of the sequences}

This section contains several computations related to the initial values of the sequences we are considering.

   \begin{lemma}   \label{lem:r-init}
        For all $k \in \N$ with $k\geqslant 2$ and for each $n \in \{1,\cdots, 2^k-2\}$ we have $r_{k,n}^0 = 0$. Moreover, \[r_{k,2^k-1}^0 = 1 , \ \  \ \ r_{k,2^k}^0 = 2^{k-1}-1, \ \ \ r_{k,2^k+1}^0= \binom{2^{k-1}}{2} + 2^{k-2}. \] \label{the r's}
    \end{lemma}
    \begin{proof}
        Since $\{1,\cdots, 2^k-1\}$ is a maximal $k$-Schreier set, no proper subset can be a maximal $k$-Schreier set. This proves that
        for each $n \in \{1,\cdots, 2^k-2\}$, we have $r_{k,n}^0 = 0$ and $r_{k,2^k-1}^0 = 1$. The equality $r_{k,2^k}^0 = 2^{k-1}-1$ follows from the fact that  $\{1,\cdots, 2^k\}\setminus\{n\}$ is a maximal $k$-Schreier set for each $n \in \{2^{k-1}+1,\cdots, 2^k-1\}$, and every such set is of this form.  To compute $r_{k,2^k+1}^0$, one easily checks that any set $F \in \mathcal{R}_{k,2^k+1}^0$ has initial segment  $\{1, \ldots, 2^{k-2}\}$ and is made by either deleting two numbers from the set $\{2^{k-1}+1, \cdots, 2^{k}\}$ or deleting one number from $\{2^{k-2}+1,\ldots,2^{k-1}\}$. 
            \end{proof}

\noindent For all $n, k \in \N$, set the following notation
\[ t_{k,n} := \sum_{i=1}^{n} r_{k,i}^0.\]    

\noindent The next lemma follows trivially from Lemma \ref{the r's}.

      \begin{lemma}   \label{the t's}
        For all $k \in \N$ with $k \geqslant 2$ and for each $n \in \{1,\cdots, 2^k-2\}$ we have $t_{k,n} = 0$. Moreover, \[t_{k,2^k-1} = 1 , \ \  \ \ t_{k,2^k} = 2^{k-1}, \ \ \ t_{k,2^k+1}= 2^{k-1} + \binom{2^{k-1}}{2} + 2^{k-2}. \]
    \end{lemma}  

      \begin{lemma}   \label{the s's}
        For all $k \in \N$, and for each $n \in \{1,\cdots, 2^k-1\}$, we have $s_{k,n} = 2^{n-1}$. Moreover, \[s_{k,2^k}^0 = 2^{2^k-1}-1 , \ \  \ \ s_{k,2^k+1} = 2^{2^k}-(2^{k-1}+1). \]
    \end{lemma}  

    \begin{proof}
        Fix $k \in\mathbb{N}$. The set $\{1,\ldots, 2^{k}-1\}$ is a maximal $k$-Schreier set. Therefore, every subset of $\{1,\ldots, 2^{k}-1\}$ is $k$-Schreier set. Consequently, for each $n \in \{1,\ldots, 2^k-1\}$, we have $s_{k,n} = 2^{n-1}$. On the other hand, $\{1,\ldots, 2^{k}\}$ is not a $k$-Schreier set, but every of its proper subset is a $k$-Schreier. Thus, $s_{k,2^k} = 2^{2^k-1}-1$ 
        
        We compute $s_{k,2^k+1}$. We will show there are exactly $2^k+1$ many subsets of $\{1,\ldots,2^k+1\}$ with maximum element $2^k+1$ that are not in $k\calS$. Let $F$ be such a set and assume the case that $F\not=\{1,\ldots,2^k+1\}$. 
        
        First observe that $\{1,\ldots,2^{k-1}\} \subset F$. Indeed, if this were not the case, then $\min E_k(F)>2^{k-1}$, which implies that $\max E_k(F)\geqslant2^{k}+1$. Hence, since $F\not \in k\mathcal{S}$, $\max F >2^k+1$, a contradiction. Therefore, $\min E_k(F)=2^{k-1}$, and since $2^k-1 \leq \max E_k(F)<2^k+1$,  it follows that $F= \{1, \ldots, 2^k+1\}\setminus \{n\}$ for each $n \in \{2^{k-1} + 1,\ldots,2^k\}$. Including the case that $F=\{1, \ldots,2^k+1\}$. We have a total of $2^k+1$ many such sets. This is the desired result.
     \end{proof}

\noindent For each $k \in \mathbb{N}$, let $(d^k_{i})_{i=0}^{2^k}$ be so that 
$$p_k(x)= \sum_{i=0}^{2^k} d^k_i x^i.$$

\begin{lemma}\label{the d's}
Let $k \in \mathbb{N}$ with $k \geqslant 2$. Then 
\[d_{2^k-2}^k = \binom{2^{k-1}}{2}  -2^{k-2}, \ \  \ \ d^k_{2^k-1}=-2^{k-1} \ \  ,  d^k_{2^k}=1 \]
\end{lemma}

\begin{proof}
We proceed by induction on $k$. Since $p_2(x)=x^4-2x^3+x -1$, the base case $k=2$ holds. Fix some integer $k$ with $k\geqslant 3$ and assume the identities hold for $k-1$. Let $m=2^{k-1}$. Then
$$p_k(x) = p_{k-1}(x(x-1)) = x^m (x-1)^m - \frac{m}{2}x^{m-1}(x-1)^{m-1} + \cdots .$$
By inspection of the leading coefficients and using the binomial theorem, we have
$$p_k(x) = x^{2m} -mx^{2m-1} + \bigg(\binom{m}{2} - \frac{m}{2}\bigg)x^{2m-2} + \cdots .$$
This completes the inductive step of the proof.
\end{proof}

The next lemma verifies that the recursions are satisfied for the initial $2^k+1$ values of the sequence. This lemma will be repeatedly used as the base case of the induction in the proof of Theorem \ref{thm:lin-rec}.

\begin{lemma}
For each $k \in \mathbb{N}$, we have
$$\sum_{i=1}^{2^k+1} d^k_{i-1} r^0_{k,i} =\sum_{i=0}^{2^k} d^k_{i} t_{k,i+1} = \sum_{i=1}^{2^k} d^k_{i} t_{k,i} = \sum_{i=0}^{2^k} d^k_{i} s_{k,i+1}=0.$$
\label{initialsums}
\end{lemma}

\begin{proof} Direct calculation using the above lemmas show that the first three sums are 0. For the final identity, notice first that $p_k(2)=1$ for each $k \in \mathbb{N}$. Thus using the identities in Lemmas \ref{the s's} and \ref{the d's}, we have
\begin{equation}
    \begin{split}
        0=p_k(2)-1 & = d^k_0 + d^k_1 2 + \cdots + d^k_{2^k-1}(2^{2^k-1}-1) + d^k_{2^k}2^{2^k}-1+d^k_{2^k-1}\\
        & = d^k_0 + d^k_1 2 + \cdots + d^k_{2^k-1}(2^{2^k-1}-1) +d^k_{2^k}(2^{2^k}-2^{k-1}-1) \\
        & = \sum_{i=0}^{2^k} d^k_{i} s_{k,i+1}\qedhere
    \end{split}
\end{equation}
\end{proof}

   \section{Pascal-like sets and recursions}
   
        A set $\{r^j_{i}:i \in \N, j\in \{0,\cdots, i-1\}\}$ is \emph{Pascal-like} if for every $i \in \mathbb{N}$ with $i\geqslant 3$ and every $j \in \{0,\cdots, i-3\}$ we have
        $r_{i}^{j} = r_{i-1}^{j} + r_{i-1}^{j+1}.$
By definition, the Pascal-like set is determined by the values of the two main diagonals $(r^{i-1}_{i})_{i=1}^\infty$ and $(r^{i-1}_{i+1})_{i=1}^\infty$. 
See \cref{fig:pascal-set-ex} for an illustration. 
In the next lemma, we express each element of the main diagonal in terms of the first column.

\begin{figure}[!t]
    \centering
    \includegraphics{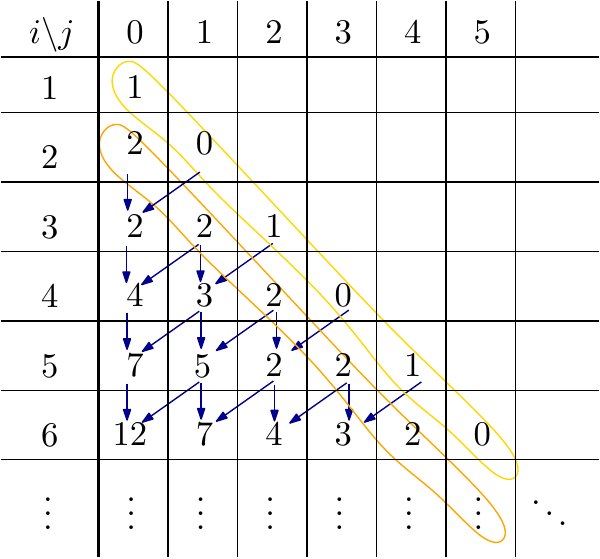}
    \caption{An initial fragment of a Pascal-like set. The cells of a table are filled with the numbers $r_i^j$, e.g.\ $r_6^1 = 7$. The set is determined by the yellow and orange diagonals.}
    \label{fig:pascal-set-ex}
\end{figure}

\begin{lemma} \label{lem:pascal-fm}
    Let $\{r_{i}^{j}:i \in \N, j\in \{0,\cdots, i-1\}\}$ be Pascal-like set. Then for each $n \in \mathbb{N}$, we have
        \begin{equation}
            r_{n}^{n-1} = \sum_{j=0}^{n-1} (-1)^{n-1-j} \binom{n-1}{j} r^0_{n+j}. \label{binomal}
        \end{equation} 
\end{lemma}
\begin{proof}
We proceed by induction. We shall prove that for each $n\in  \mathbb{N}$, any Pascal-like set satisfies (\ref{binomal}). For $n = 1$, the assertion is clearly satisfied. Fix $n \geqslant 2$ and assume that (\ref{binomal}) holds replacing $n$ with $n-1$ for any Pascal-like set. Fix a Pascal-like set $\{r_{i}^{j}:i \in \N, j\in \{1,\cdots, i-1\}\}$ and define $b_{i}^{j} = r_{i+1}^{j+1}$. Then $\{b_{i}^{j}:i \in \N, j\in \{0,\cdots, i-1\}\}$ is also a Pascal-like set. By the inductive assumption, we have
    \begin{align*}
        r_{n}^{n-1} = b_{n-1}^{n-2} &= \sum_{j=0}^{n-2} (-1)^{n-2-j} \binom{n-2}{j} b_{n-1+j}^{0} = \sum_{j=0}^{n-2} (-1)^{n-2-j} \binom{n-2}{j} r^1_{n+j} \\
        &= \sum_{j=0}^{n-2} (-1)^{n-2-j} \binom{n-2}{j} (r^0_{n+j+1} - r^0_{n+j}) \\
        &= \sum_{j=1}^{n-1} (-1)^{n-1-j} \binom{n-2}{j-1} r^0_{n+j} + \sum_{j=0}^{n-2} (-1)^{n-1-j} \binom{n-2}{j} r^0_{n+j} \\
        &= \sum_{j=0}^{n-1} (-1)^{n-1-j} \binom{n-1}{j} r^0_{n+j}.
    \end{align*}
    Notice that, in the fourth equality, we used the definition of a Pascal-like set.
\end{proof}

The properties of Pascal-like sets allow transferring recursive relations between diagonals and columns.
In \cref{any diag} we show how recursive relations are transferred between diagonals and in \cref{got it right} we show how recursive relations are transferred from the main diagonals to the first column (see \cref{fig:weighted} for intuition).
Next, in \cref{mainlemma}, we relate the recursive relation on the main diagonals with the partial sums of the first column.

\begin{lemma}
Let $\{r_{i}^{j}:i \in \N, j\in \{0,\cdots, i-1\}\}$ be a Pascal-like set. If $(r_{i}^{i-1})_{i=1}^\infty$ and $(r_{i+1}^{i-1})_{i=1}^\infty$ both satisfy a linear recurrence with characteristic polynomial $p(x)$, then for each integer $\ell$ with $\ell\geqslant 2$, the sequence $(r_{\ell+i}^{i-1})_{i=1}^\infty$ satisfies a linear recurrence with characteristic polynomial $p(x)$. \label{any diag}
\end{lemma}

\begin{proof} We sketch the easy proof.
The Pascal-like property yields that $r^{i-1}_{i+2} =r^{i-1}_{i+1} + r^{i}_{i+1}$ for each $i \in \mathbb{N}$. It is clear that $(r^{i-1}_{i+2})_{i=1}^\infty$ must satisfy the same recursions as both $(r_{i}^{i-1})_{i=1}^\infty$ and $(r_{i+1}^{i-1})_{i=1}^\infty$. Using this observation as the base case for an inductive argument, the general case follows. 
\end{proof}

\begin{figure}[!t]
    \centering
    \includegraphics{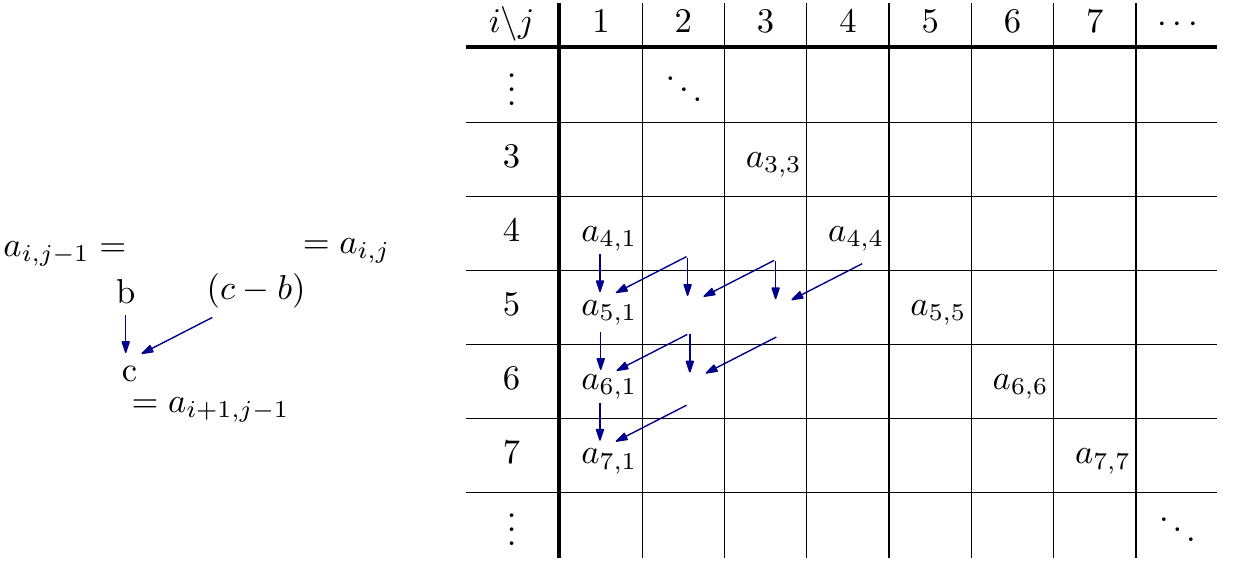}
    \caption{The goal is to relate the main diagonal entries with the first column entries. Knowing the entries $a_{4,1},a_{5,1},a_{6,1},a_{7,1}$ we can restore the value of $a_{4,4}$ using the equation $a_{i,j} = a_{i+1,j-1} - a_{i,j-1}$, which is illustrated on the left hand-side.}
    \label{fig:weighted}
\end{figure}

\begin{lemma}
Suppose that $\{r^j_{i}:i \in \N, j\in \{0,\cdots, i-1\}\}$ is \emph{Pascal-like} and $(r_{i}^{i-1})_{i=1}^\infty$ and $(r_{i+1}^{i-1})_{i=1}^\infty$ both satisfy a linear recurrence with characteristic polynomial $p(x)$. Then the sequence $(r^0_{i})_{i=1}^\infty$ satisfies the linear recurrence with characteristic polynomial $p(x(x-1))$.  
\label{got it right}
\end{lemma}

\begin{proof} 
Let $p(x) = c_kx^k + c_{k-1}x^{k-1} + \cdots c_1x+ c_0$. Let $(d_i)_{i=0}^{2k}$ be defined by $p(x(x-1))=\sum_{i=0}^{2k}d_{i}x^{i}$.
It suffices to show that for each non-negative integer $\ell$
\begin{equation}
    \sum_{i=1}^{2k+1}d_{i-1}r^0_{\ell+i} =0. \label{yes zero}
\end{equation}
Fix some non-negative integer $\ell$. It follows from the definition that $\{r^j_{\ell+i}: i \in \mathbb{N}, j \in \{0, \cdots,i-1\}\}$ is Pascal-like. Therefore, by Lemma \ref{any diag}, we know that $\sum_{i=0}^kc_ir^i_{\ell+i+1} = 0$. Using Lemma \ref{lem:pascal-fm}, we rewrite this sum as 
\begin{equation}
0=\sum_{i=0}^kc_ir^i_{\ell+i+1}= \sum_{i=0}^k c_i \sum_{j=0}^{i} (-1)^{i-j} \binom{i}{j} r^0_{\ell+i+1+j}. \label{its zero} 
\end{equation}
Since (\ref{its zero}) is satisfied, we know there are coefficients  $(d'_i)_{i=0}^{2k}$ so that $\sum_{i=1}^{2k+1}d'_{i-1}r^0_{\ell+i}=0$ (for each non-negative integer $\ell$). We wish to show that $d'_i=d_i$. Rather than a messy calculation, a  simple way to see this is by converting to the characteristic polynomial and matching the coefficients. That is, replace $r^0_{\ell+\alpha}$ by $x^{\alpha-1}$ in (\ref{its zero}) and apply the binomial theorem to obtain
$$\sum_{i=0}^k c_i \sum_{j=0}^{i} (-1)^{i-j} \binom{i}{j} x^{i+j}= \sum_{i=0}^k c_i x^i(x-1)^i= p(x(x-1))=\sum_{i=0}^{2k}d_{i}x^{i}.$$
Therefore $d'_i=d_i$ as desired.
\end{proof}

\begin{lemma}
Let 
$\{r^j_{i}:i \in \N, j\in \{0,\cdots, i-1\}\}$ be Pascal-like so that $(r_{i}^{i-1})_{i=1}^\infty$ and $(r_{i+1}^{i-1})_{i=1}^\infty$ both satisfy a linear recurrence with characteristic polynomial $p(x)$. Let $(d_i)_{i=0}^{2k}$ be the coefficients of the polynomial $p(x(x-1))$ and $t_j = \sum_{i=1}^j r^0_{i}$ for $j\in \mathbb{N}$. Suppose
 \begin{equation}
   \sum_{j=1}^{2k+1}d_{j-1} t_j=0. \label{n=0 case}
  \end{equation}
 Then $(t_j)_{j=1}^\infty$ is a linear recurrence relation with characteristic polynomial $p(x(x-1))$.
 \label{mainlemma}
\end{lemma}

\begin{proof}
Set all the notation as in the statement of the lemma. We proceed by induction to show that for each non-negative integer $\ell$. 
\begin{equation}   \sum_{j=1}^{2k+1}d_{j-1} t_{\ell+j}=0 
\label{inductionpartial}
\end{equation}
The base case of the induction $\ell=0$ is assumed -- (\ref{n=0 case}). 

Suppose that for some $\ell \geqslant 0$, (\ref{inductionpartial}) holds. We will show that (\ref{inductionpartial}) holds for $\ell+1$. 
By \cref{got it right}, $(r^0_i)_{i=1}^\infty$ satisfies the recurrence relation with characteristic polynomial $p(x(x-1))$. Using this fact and the induction hypothesis, we have the following 
\begin{equation}
    \begin{split}
     \sum_{j=1}^{2k+1}d_{j-1} t_{\ell+1+j} & = \sum_{i=1}^{2k+1}d_{j-1} (t_{\ell+1+j}-t_{\ell+j}) +\sum_{i=1}^{2k+1}d_{j-1} t_{\ell+j}\\
     & =\sum_{j=1}^{2k+1}d_{j-1} r^0_{\ell+1+j} +\sum_{j=1}^{2k+1}d_{j-1} t_{\ell+j}=0 +0 =0   
    \end{split}
\end{equation}
This completes the proof.
\end{proof}

\section{Proof of Theorem \ref{thm:lin-rec}}

In the previous section, we made some general observations on Pascal-like sets.
Now, in order to use them, we show that the numbers $r_{k,n}^d$ form Pascal-like sets -- as in (O1).
Next, in order to formally prove (O3), we prove that (\ref{bigdeal}) holds in general.
    
 \begin{lemma} \label{move tables}
     Let $n,k \in \mathbb{N}$. Then
          \[r_{k+1,n+1}^{n}  = t_{k,n}:=  \sum_{i=1}^{n} r_{k,i}^0.\]
    \end{lemma}
    \begin{proof}

        First, recall that the upper index $0$ in the expression $r_{k,n}^0$ indicates counting maximal $k\mathcal{S}$ sets. Observe that every element of  $\calR_{k+1,n+1}^{n}$ is the union of two sets: The first is a maximal $k\calS$ set with a maximum at most $n-1$. The second is $\{n+1\}$. Consider the following map
            \[ \calR_{k+1,n+1}^{n} \ni F \mapsto F \backslash \{n+1\} \in \bigcup_{i=1}^{n} \calR_{k,i}^0.\]
        The map is clearly bijective ("$i$" can be interpreted as the second largest element of $F$) and the families under the union in the co-domain are pairwise disjoint.
              \end{proof}        

    \begin{lemma} \label{lem:pascal-rec} 
        Let $k \in \N$, $n \in \N$ with $n\geqslant 3$, and $d \in \{0,\cdots, n-3\}$. Then 
            \[ r_{k,n}^d = r_{k,n-1}^d + r_{k,n-1}^{d+1}. \]
        Hence, the set $\{r_{k,i}^{j}:i\in \mathbb{N}, j\in \{0,\cdots i-1\}\}$ is a Pascal-like. Moreover, $r^{j-1}_{k,j+1}=0$ for each $j \in \mathbb{N}$.
    \end{lemma}
    \begin{proof}
         Define
        \begin{align*}
            \calT_1 := \{F \in \calR_{k,n}^{d} \ : \ n-1 \in F\}, &  \ \ \ \ \ \ \calT_2 := \{F \in \calR_{k,n}^{d} \ : \ n-1 \notin F\},\\
            \calT_1' := \{F \backslash \{n\} \ : \ F \in \calT_1 \}, & \ \ \ \ \ \ \calT_2' := \{F \backslash \{n\} \cup \{n-1\} \ : \ F \in \calT_2 \}.
        \end{align*}
        Clearly, we have $|\calT_1|=|\calT_1'|$ and $|\calT_2|=|\calT_2'|$. Hence, $r_{k,n}^d = |\calT_1| + |\calT_2| = |\calT_1'| + |\calT_2'|$. Consider some $F \in \calR_{k,n}^d$, we claim that $|E_k(F)| > 1$. If $E_k(F) = \{n\}$, then $d_k(F) = n - 1$, but we assumed that $d < n-1$, which is a contradiction. Comparing the sets $E_k(F)$ and $E_k(F')$, where $F'$ is the set $F$ modified in one of two presented ways, it can be easily checked that $\calT_1' = \calR_{k,n-1}^{d+1}$ and $\calT_2' = \calR_{k,n-1}^d$. This concludes the proof of the first part.

        We will show that for each $n \in \mathbb{N}$, we have $\mathcal{R}_{k,n+1}^{n-1}=\emptyset$. We know that if $F \in \mathcal{R}_{k,n+1}^{n-1}$ we have $n+1 \in E_k(F)$ and, by definition, $d_k(F)=n-1$. We obtain contradictions for all possible values of $\min E_k(F)$. If $\min E_k(F)\leqslant n$, using that $n+1\in E_k(F)$, we have $d_k(F)\leqslant n-2$. If $n+1=\min E_k(F)$, then $d_k(F)= n$.
    \end{proof}
    

    
    \begin{lemma}\label{lem:s=2s+r}
    Let $k,n \in \N$ with $n \geqslant 2$. We have
    $$s_{k,n}=2s_{k,n-1}-r_{k,n-1}^0.$$
    \end{lemma}
    \begin{proof}
         Define
        \begin{align*}
            \calT_1 &:= \{F \cup \{n\} \ : \ F \in (k\calS)^{n-1} \backslash \calR_{k,n-1}^0 \}, \\
            \calT_2 &:= \{F \backslash \{n-1\} \cup \{n\} \ : \ F \in (k\calS)^{n-1} \}.
        \end{align*}
        We claim that $(k\calS)^n = \calT_1 \cup \calT_2$. Then, by the fact that $\calT_1 \cap \calT_2 = \emptyset$, we obtain 
        \[s_{k,n} = |(k\calS)^n| = |\calT_1| + |\calT_2| = (s_{k,n-1} - r_{k,n-1}^0) + s_{k,n-1}.\]
        
        For each $G \in (k\calS)^n$ we have $F = G \backslash \{n\} \cup \{n - 1\} \in (k\calS)^{n-1}$.
        If $n-1 \in G$, $F=G\setminus \{n\}$. Since $G \in (k\mathcal{S})^n$, $F$ is not a maximal $k$-Schreier set. Hence, $G \in T_1$. If $n-1 \not \in G$, $G \in T_2$. It follows that $(k\calS)^n \subset \calT_1 \cup \calT_2$. The inclusion $\calT_2 \subset (k\calS)^n$ is obvious. It suffices to prove that $\calT_1 \subset (k\calS)^n$, however, this is almost obvious -- adding one element larger than the maximum to a non-maximal $k$-Schreier produces a $k$-Schreier set.
    \end{proof}

We are now ready to proceed to the proof of Theorem \ref{thm:lin-rec}.

\begin{proof}[Proof of Theorem \ref{thm:lin-rec}]

We claim that for each $k \in \mathbb{N}$, the sequence $(s_{k,n})$ satisfies a linear recurrence with characteristic polynomial $p_k(x)$. This claim has already been established in previous work for $k=1$; however, we will not use this statement directly. Instead, we prove the following statement holds for each natural number $k$.
\begin{itemize}
    \item[(P$_k$):] The sequences  $(r^0_{k,i})_{i=1}^\infty, (t_{k,i})_{i=1}^\infty$ satisfy a linear recurrence with characteristic polynomial $p_k(x)$.
\end{itemize}
Consider the base case $k=1$. Note that $r_{1,n}^{n-1}=1$ for each $n\in \mathbb{N}$ since $F \in \mathcal{R}_{1,n}^{n-1}$ if and only if $F=\{n\}$. Therefore $(r_{1,n}^{n-1})$ satisfies a linear recurrence with characteristic polynomial $p_0(x)=x-1$ and the same holds for $(r_{1,n+1}^{n-1})_{n=1}^\infty$ since this sequence is identically 0 (Lemma \ref{lem:pascal-rec}). Therefore by Lemma \ref{got it right}, $(r^0_{1,i})_{i=1}^\infty$ satisfies a linear recurrence with characteristic polynomial $p_0(x(x-1))=x(x-1)-1=p_1(x)$. Notice that (\ref{n=0 case}) is satisfied by Lemma \ref{initialsums} for $k=1$. Therefore we may apply Lemma \ref{mainlemma} to see that $(t_{1,i})_{i=1}^\infty$ satisfies a linear recurrence with characteristic polynomial $p_1(x)$. Therefore (P$_1$) holds.

Fix a positive integer $k$ and assume that (P$_k$)
holds; we will prove that (P$_{k+1}$) holds. 
By Lemma \ref{move tables}, we have that
$t_{k,i}=r_{k+1,i+1}^i$ for each $i \in \mathbb{N}$. Therefore, $(r_{k+1,i+1}^i)_{i=1}^\infty$ satisfies a linear recurrence with characteristic polynomial $p_k(x)$. We want to show that $(r_{k+1,i}^{i-1})_{i=1}^\infty$ satisfies a linear recurrence with characteristic polynomial $p_k(x)$; this is just the sequence in the previous sentence starting with $r_{k+1,1}^0$. Notice that $r_{k+1,1}^0=0$. Therefore, by Lemma \ref{initialsums},
$$\sum_{i=0}^{2^k} d^k_{i}r_{k+1,i+1}^i = \sum_{i=1}^{2^k} d^k_{i}t_{k,i}=0.$$
Therefore the initial conditions are satisfied and so $(r_{k+1,i}^{i-1})_{i=1}^\infty$ satisfies a linear recurrence with characteristic polynomial $p_k(x)$.

By Lemma \ref{lem:pascal-rec}, we have $r_{k+1,i+1}^{i-1}=0$ for each $i \in \mathbb{N}$. Thus we may apply Lemma \ref{got it right} to conclude that $(r^0_{k+1,i})_{i=1}^\infty$ satisfies a linear recurrence with characteristic polynomial $p_{k+1}(x)$. In addition, by Lemma \ref{initialsums}, the assumptions of Lemma \ref{mainlemma} are satisfied, and obtain that $(t_{k+1,i})_{i=1}^\infty$ satisfies a linear recurrence with characteristic polynomial $p_{k+1}(x)$. This concludes the inductive step.

Fix $k\in \mathbb{N}$. It remains to prove that $(s_{k,n})_{n=1}^\infty$ satisfies a linear recurrence with characteristic polynomial $p_{k}(x)$. It suffices to show that for each non-negative integer $\ell$, we have
\begin{equation}
    \sum_{i=0}^{2^k}d^k_{i} s_{k,\ell+i+1} =0.\label{almost done}
\end{equation}
By Lemma \ref{initialsums}, the $\ell=0$ case holds. Suppose (\ref{almost done}) holds for some non-negative $\ell$. Using Lemma~\ref{lem:s=2s+r}, and the statement (P$_k$), we have
\begin{equation}
    \sum_{i=0}^{2^k}d^k_{i} s_{k,\ell+i+2} = 2\sum_{i=0}^{2^k}d^k_{i} s_{k,\ell+i+1 } - \sum_{i=0}^{2^k} d^k_i r^0_{k,\ell+i+1} =0-0.
\end{equation}
This finishes the proof of the theorem. Note that Corollary \ref{only cor} follows from the initial induction in the proof -- statements (P$_k$).
\end{proof}

\section{Concluding remarks}
 
It seems there are many interesting open problems similar to the one we studied. For example, there is another important regular family $\calS_2$, which is the convolution of $\calS$ with itself. That is, a subset of natural numbers $F$ is in $\calS_2$ if there exist disjoint nonempty sets $E_1,\dots,E_\ell \in \calS$ such that
        \[ \bigcup_{i=1}^\ell E_i = F, \ \ \mathrm{and} \ \ \set{\min E_i : i \in \{1, \ldots, \ell\} } \in \calS. \]
    The family $\calS_2$ also appears naturally in Banach space theory \cite{Alspach-Argyros}.
    
    \begin{problem}
        Is the sequence $(|\calS_2^n|)_{n=1}^\infty$ a linear recurrence sequence? If yes, then what is the recursive relation?
    \end{problem}

	\bibliographystyle{abbrv}
	\bibliography{bib_source1}

\begin{thebibliography}{10}

\bibitem{Alspach-Argyros}
D.~E. Alspach and S.~Argyros.
\newblock Complexity of weakly null sequences.
\newblock {\em Dissertationes Math. (Rozprawy Mat.)}, 321:44, 1992.

\bibitem{AT-Book}
S.~A. Argyros and S.~Todorcevic.
\newblock {\em Ramsey methods in analysis}.
\newblock Advanced Courses in Mathematics. CRM Barcelona. Birkh\"{a}user
  Verlag, Basel, 2005.

\bibitem{Hung3}
K.~Beanland, H.~V. Chu, and C.~E. Finch-Smith.
\newblock Generalized {S}chreier sets, linear recurrence relation, and
  {T}ur\'{a}n graphs.
\newblock {\em Fibonacci Quart.}, 60(4):352--356, 2022.

\bibitem{Beanland-Hodor}
K.~Beanland and J.~Hodor.
\newblock The depth of {T}sirelson's norm.
\newblock {\em arXiv}, 2023.
\newblock \href{https://arxiv.org/abs/2306.10344}{arXiv:2306.10344}.

\bibitem{Hung4}
H.~V. Chu.
\newblock The {F}ibonacci sequence and {S}chreier-{Z}eckendorf sets.
\newblock {\em J. Integer Seq.}, 22(6):Art. 19.6.5, 12, 2019.

\bibitem{Hung5}
H.~V. Chu.
\newblock Various sequences from counting subsets.
\newblock {\em Fibonacci Quart.}, 59(2):152--157, 2021.

\bibitem{Hung1}
H.~V. Chu.
\newblock On a relation between {S}chreier-type sets and a modification of
  {T}ur\'{a}n graphs.
\newblock {\em Integers}, 23:Paper No. A20, 10, 2023.

\bibitem{Hung6}
H.~V. Chu, S.~J. Miller, and Z.~Xiang.
\newblock Higher order {F}ibonacci sequences from generalized {S}chreier sets.
\newblock {\em Fibonacci Quart.}, 58(3):249--253, 2020.

\bibitem{Odell-Ordinal}
E.~Odell.
\newblock Ordinal indices in {B}anach spaces.
\newblock {\em Extracta Math.}, 19(1):93--125, 2004.

\bibitem{Schreier1930}
J.~Schreier.
\newblock Ein {G}egenbeispiel zur {T}heorie der schwachen {K}onvergenz.
\newblock {\em Studia Mathematica}, 2(1):58--62, 1930.

\end{thebibliography}

\end{document}